\documentclass[12pt]{article}
\usepackage{amsfonts}
\usepackage{amsmath,amscd,amsthm,a4,amssymb}
\usepackage{amsmath,amscd,amsthm,a4,amssymb}

\newtheorem{theorem}{Theorem}
{}
\newtheorem{corollary}{Corollary}
{}

{}

{}

{}
\theoremstyle{plain}
{}

{}

{}

{}

{}

{}

{}

{}

{}

{}

{}

{}

{}

{}

{}

{}

\begin{document}
	\begin{center}
		{\Large \bf{Rotational Surfaces with Pointwise 1-Type Gauss Map in Pseudo
				Euclidean Space $\mathbb{E}_{2}^{4}$ }}
	\end{center}
	\centerline{\large Ferdag KAHRAMAN AKSOYAK  $^{1}$, Yusuf YAYLI $^{2}${\footnotetext{
				{E-mail: $^{1}$ferdag.aksoyak@ahievran.edu.tr(F. Kahraman Aksoyak  $^{2}$yayli@science.ankara.edu.tr (Y.Yayli)}} }}
	
	\centerline{\it $^{1}$Ahi Evran University, Division of Elementary Mathematics Education
		Kirsehir, Turkey}
	\centerline{\it $^{2}$Ankara University, Department of Mathematics,
		Ankara, Turkey}

\begin{abstract}
In this paper, we study rotational surfaces of elliptic, hyperbolic and
parabolic type with pointwise 1-type Gauss map which have spacelike profile curve in four dimensional pseudo
Euclidean space $\mathbb{E}_{2}^{4}$ and obtain some characterizations for
these rotational surfaces to have pointwise 1-type Gauss map.
\end{abstract}

\begin{quote}\small
\end{quote}
\begin{quote}\small
\end{quote}

\section{Introduction}

The Gauss map $G$ of a submanifold $M$ into $G(n,m)$ in $\wedge ^{n}\mathbb{E%
}_{s}^{m},$ where $G(n,m)$ is the Grassmannian manifold consisting of all
oriented $n-$planes through the origin of $\mathbb{E}_{s}^{m}$ and $\wedge
^{n}\mathbb{E}_{s}^{m}$ is the vector space obtained by the exterior product
of $n$ vectors in $\mathbb{E}_{s}^{m}$ is a smooth map which carries a point 
$p$ in $M$ into the oriented $n-$plane in $\mathbb{E}_{s}^{m}$ obtained from
parallel translation of the tangent space of $M$ at $p$ in $\mathbb{E}%
_{s}^{m}.$ Since the vector space $\wedge ^{n}\mathbb{E}_{s}^{m}$ identify
with a semi-Euclidean space $\mathbb{E}_{t}^{N}$ for some positive integer $%
t,$ where $N=\left( 
\begin{array}{c}
m \\ 
n%
\end{array}%
\right) ,$ the Gauss map is defined by $G:M\rightarrow G(n,m)\subset \mathbb{%
	E}_{t}^{N},$ $G(p)=\left( e_{1}\wedge ...\wedge e_{n}\right) \left( p\right) 
$. The notion of submanifolds with finite type Gauss map was introduced by
B. Y.Chen and P.Piccinni in 1987 \cite{chen} and after then many works were
done about this topic, especially 1-type Gauss map and 2- type Gauss map.

If a submanifold $M$ of a Euclidean space or pseudo-Euclidean space has
1-type Gauss map $G$, then $G$ satisfies 
\begin{equation*}
\Delta G=\lambda \left( G+C\right)
\end{equation*}
for some $\lambda \in \mathbb{R}$ and some constant vector $C.$

On the other hand the Laplacian of the Gauss map of some typical well-known
surfaces satisfy the form of%
\begin{equation}
\Delta G=f\left( G+C\right) 
\end{equation}%
for some smooth function $f$ on $M$ and some constant vector $C.$ A
submanifold of a Euclidean space or pseudo-Euclidean space is said to have
pointwise 1-type Gauss map, if its Gauss map satisfies (1) for some smooth
function $f$ on $M$ and some constant vector $C.$ If the vector $C$\ in (1)
is zero, a submanifold with pointwise 1-type Gauss map is said to be of the
first kind, otherwise it is said to be of the second kind.

A lot of papers were recently published about rotational surfaces with
pointwise 1-type Gauss map in four dimensional Euclidean and pseudo
Euclidean space in \cite{aksoyak 1},\cite{arslan 1},\cite{arslan 2}, \cite%
{dursun 1}, \cite{dursun 2} \cite{kim}.Timelike and spacelike rotational
surfaces of elliptic, hyperbolic and parabolic types in Minkowski space $%
\mathbb{E}_{1}^{4}$ with pointwise 1-type Gauss map were studied in \cite%
{bektas 1, bektas 2}. Aksoyak and Yayl\i \ in \cite{aksoyak 2} studied boost
invariant surfaces (rotational surfaces of hyperbolic type) with pointwise
1-type Gauss map in Minkowski space $\mathbb{E}_{1}^{4}$. They gave a
characterization for flat boost invariant surfaces with pointwise 1-type
Gauss map. Also they obtain some results for boost invariant marginally
trapped surfaces with pointwise 1-type Gauss map. Ganchev and Milousheva in 
\cite{milo} defined three types of rotational surfaces with two dimensional
axis rotational surfaces of elliptic, hyperbolic and parabolic type in
pseudo Euclidean space $\mathbb{E}_{2}^{4}$. They classify all rotational
marginally trapped surfaces of elliptic, hyperbolic and parabolic type,
respectively.

In this paper, we study rotational surfaces of elliptic, hyperbolic and
parabolic type with pointwise 1-type Gauss map which have spacelike profile
curve in four dimensional pseudo Euclidean space and give all
classifications of flat rotational surfaces of elliptic, hyperbolic and
parabolic type with pointwise 1-type Gauss map.

\section{Preliminaries}

Let $\mathbb{E}_{s}^{m}$ be the $m-$dimensional pseudo-Euclidean space with
signature $(s,m-s)$. Then the metric tensor $g$ in $\mathbb{E}_{s}^{m}$ has
the form 
\begin{equation*}
g=\sum \limits_{i=1}^{m-s}\left( dx_{i}\right) ^{2}-\sum
\limits_{i=m-s+1}^{m}\left( dx_{i}\right) ^{2}
\end{equation*}%
where $(x_{1},...,x_{m})$ is a standard rectangular coordinate system in $%
\mathbb{E}_{s}^{m}.$

A vector $v$ is called spacelike (resp., timelike) if $\left \langle
v,v\right \rangle >0$ (resp., $\left \langle v,v\right \rangle <0$). Avector 
$v$ is called lightlike if it $v\neq 0$ and $\left \langle v,v\right \rangle
=0,$ where $\left \langle ,\right \rangle $ is indefinite inner scalar
product with respect to $g.$

Let $M$ be an $n-$dimensional pseudo-Riemannian submanifold of a $m-$%
dimensional pseudo-Euclidean space $\mathbb{E}_{s}^{m}$ and denote by $%
\tilde{\nabla}$ and $\nabla $ Levi-Civita connections of $\mathbb{E}_{s}^{m}$
and $M$ $,$ respectively. We choose local orthonormal frame $\left \{
e_{1},...,e_{n},e_{n+1},...,e_{m}\right \} $ on $M$ with $\varepsilon
_{A}=\left \langle e_{A},e_{A}\right \rangle =\pm 1$ such that $e_{1},$...,$%
e_{n}$ are tangent to $M$\ and $e_{n+1},$...,$e_{m}$ are normal to $M.$ We
use the following convention on the ranges of indices: $1\leq i,j,k,$...$%
\leq n$, $n+1\leq r,s,t,$...$\leq m$, $1\leq A,B,C,$...$\leq m.$

Denote by $\omega _{A}$ the dual-1 form of $e_{A}$ such that $\omega
_{A}\left( X\right) =\left \langle e_{A},X\right \rangle $ and $\omega _{AB}$
the connection forms defined by%
\begin{equation*}
de_{A}=\sum \limits_{B}\varepsilon _{B}\omega _{AB}e_{B},\text{ \  \ }\omega
_{AB}+\omega _{BA}=0.
\end{equation*}%
Then the formulas of Gauss and Weingarten are given by%
\begin{equation*}
\tilde{\nabla}_{e_{k}}^{e_{i}}=\sum \limits_{j=1}^{n}\varepsilon _{j}\omega
_{ij}\left( e_{k}\right) e_{j}+\sum \limits_{r=n+1}^{m}\varepsilon
_{r}h_{ik}^{r}e_{r}
\end{equation*}%
and%
\begin{equation*}
\tilde{\nabla}_{e_{k}}^{e_{s}}=-\sum \limits_{j=1}^{n}\varepsilon
_{j}h_{kj}^{s}e_{j}+D_{e_{k}}^{e_{s}},\text{ \ }D_{e_{k}}^{e_{s}}=\sum%
\limits_{r=n+1}^{m}\varepsilon _{r}\omega _{sr}\left( e_{k}\right) e_{r},
\end{equation*}%
where $D$ is the normal connection, $h_{ik}^{r}$ the coefficients of the
second fundamental form $h.$

For any real function $f$ on $M,$ the Laplacian operator of $M$ with respect
to induced metric is given by 
\begin{equation}
\Delta f=-\varepsilon _{i}\sum \limits_{i}\left( \tilde{\nabla}_{e_{i}}%
\tilde{\nabla}_{e_{i}}f-\tilde{\nabla}_{\nabla _{e_{i}}^{e_{i}}}f\right) .
\end{equation}%
The mean curvature vector $H$ and the Gaussian curvature $K$of $M$ in $%
\mathbb{E}_{s}^{m}$ are defined by 
\begin{equation}
H=\frac{1}{n}\sum \limits_{s=n+1}^{m}\sum \limits_{i=1}^{n}\varepsilon
_{i}\varepsilon _{s}h_{ii}^{s}e_{s}
\end{equation}%
and%
\begin{equation}
K=\sum \limits_{s=n+1}^{m}\varepsilon _{s}\left(
h_{11}^{s}h_{22}^{s}-h_{12}^{s}h_{21}^{s}\right) ,
\end{equation}%
respectively. We recall that a surface $M$ is called minimal if its mean
curvature vector vanishes identically, i.e. $H=0.$ If the mean curvature
vector satisfies $DH=0,$ then the surface $M$ is said to have parallel mean
curvature vector. Also if Gaussian curvature of $M$ vanishes identically,
i.e. $K=0,$ the surface $M$ is called flat.

\section{Rotational Surfaces with Pointwise 1-Type Gauss Map in $\mathbb{E}%
	_{2}^{4}$}

In this section, we consider rotational surfaces of elliptic, hyperbolic and
parabolic type in four dimensional pseudo-Euclidean space $\mathbb{E}%
_{2}^{4} $ which are defined by Ganchev and Milousheva in \cite{milo} and
investigate these rotational surfaces with pointwise 1-type Gauss map.

Denote by $\left \{ \epsilon _{1},\epsilon _{2},\epsilon _{3},\epsilon
_{4}\right \} $ the standart orthonormal basis of $\mathbb{E}_{2}^{4},$
i.e., $\epsilon _{1}=(1,0,0,0),$ $\epsilon _{2}=(0,1,0,0),$ $\epsilon
_{3}=(0,0,1,0)$ and $\epsilon _{4}=(0,0,0,1),$ where $\left \langle \epsilon
_{1},\epsilon _{1}\right \rangle =\left \langle \epsilon _{2},\epsilon
_{2}\right \rangle =1,$ $\left \langle \epsilon _{3},\epsilon
_{3}\right
\rangle =\left \langle \epsilon _{4},\epsilon _{4}\right \rangle
=-1. $

\subsection{Rotational surfaces of elliptic type with pointwise 1-type Gauss
	map in $\mathbb{E}_{2}^{4}$}

In this subsection, firstly we consider the rotational surfaces of elliptic
type with harmonic Gauss map. Further we give a characterization of the flat
rotational surfaces of elliptic type with pointwise 1-type Gauss map and
obtain a relationship for non-minimal these surfaces with parallel mean
curvature vector and pointwise 1-type Gauss map of the first kind.

Rotational surface of elliptic type $M_{1}$ is defined by%
\begin{equation*}
\varphi \left( t,s\right) =%
\begin{pmatrix}
1 & 0 & 0 & 0 \\ 
0 & 1 & 0 & 0 \\ 
0 & 0 & \cos t & -\sin t \\ 
0 & 0 & \sin t & \cos t%
\end{pmatrix}%
\left( 
\begin{array}{c}
x_{1}(s) \\ 
x_{2}(s) \\ 
x_{3}(s) \\ 
0%
\end{array}%
\right)
\end{equation*}%
\begin{equation}
M_{1}:\text{ }\varphi \left( t,s\right) =\left(
x_{1}(s),x_{2}(s),x_{3}(s)\cos t,x_{3}(s)\sin t\right) ,
\end{equation}%
where the surface $M_{1}$ is obtained by the rotation of the curve $%
x(s)=(x_{1}(s),x_{2}(s),x_{3}(s),0)$ about the two dimensional Euclidean
plane span$\left \{ \epsilon _{1},\epsilon _{2}\right \} .$ Let the profile
curve of $M_{1}$ be unit speed spacelike curve. In that case $\left(
x_{1}{}^{\prime }(s)\right) ^{2}+\left( x_{2}{}^{\prime }(s)\right)
^{2}-\left( x_{3}{}^{\prime }(s)\right) ^{2}=1$. We suppose that $%
x_{3}(s)>0. $ The moving frame field $\left \{
e_{1},e_{2},e_{3},e_{4}\right
\} $ on $M_{1} $\ is determined as follows: 
\begin{eqnarray*}
	e_{1} &=&\left( x_{1}{}^{\prime }(s),x_{2}{}^{\prime }(s),x_{3}{}^{\prime
	}(s)\cos t,x_{3}{}^{\prime }(s)\sin t\right) , \\
	e_{2} &=&\left( 0,0,-\sin t,\cos t\right) , \\
	e_{3} &=&\frac{1}{\sqrt{1+\left( x_{3}{}^{\prime }\right) ^{2}}}\left(
	-x_{2}{}^{\prime }(s),x_{1}{}^{\prime }(s),0,0\right) , \\
	e_{4} &=&\frac{1}{\sqrt{1+\left( x_{3}{}^{\prime }\right) ^{2}}}\left(
	x_{3}{}^{\prime }(s)x_{1}{}^{\prime }(s),x_{3}{}^{\prime }(s)x_{2}{}^{\prime
	}(s),(1+\left( x_{3}{}^{\prime }\right) ^{2})\cos t,(1+\left(
	x_{3}{}^{\prime }\right) ^{2})\sin t\right) ,
\end{eqnarray*}%
where $e_{1},e_{2}$ and $e_{3},e_{4}$ are tangent vector fields and normal
vector fields to $M_{1},$ respectively.Then it is easily seen that 
\begin{equation*}
\left \langle e_{1},e_{1}\right \rangle =\left \langle e_{3},e_{3}\right
\rangle =1,\text{ }\left \langle e_{2},e_{2}\right \rangle =\left \langle
e_{4},e_{4}\right \rangle =-1.
\end{equation*}%
We have the dual 1-forms as: 
\begin{equation*}
\omega _{1}=ds\text{ \  \  \  \ and \  \  \  \ }\omega _{2}=-x_{3}(s)dt.
\end{equation*}%
After some computations, the components of the second fundamental form and
the connection forms are given as follows:%
\begin{eqnarray}
h_{11}^{3} &=&-d(s),\ h_{12}^{3}=0,\ h_{22}^{3}=0, \\
h_{11}^{4} &=&-c(s),\text{ \ }h_{12}^{4}=0,\text{ \ }h_{22}^{4}=b(s)  \notag
\end{eqnarray}%
and%
\begin{eqnarray*}
	\omega _{12} &=&a(s)b(s)\omega _{2},\text{ \  \ }\omega _{13}=-d(s)\omega
	_{1},\text{ \  \ }\omega _{14}=-c(s)\omega _{1}, \\
	\omega _{23} &=&0,\text{ \  \ }\omega _{24}=-b(s)\omega _{2},\text{ \  \ }%
	\omega _{34}=a(s)d(s)\omega _{1}.
\end{eqnarray*}
The covariant differentiations with respect to $e_{1}$ and $e_{2}$ are
computed as: 
\begin{eqnarray}
\tilde{\nabla}_{e_{1}}e_{1} &=&-d(s)e_{3}+c(s)e_{4}, \\
\tilde{\nabla}_{e_{2}}e_{1} &=&a(s)b(s)e_{2},  \notag \\
\tilde{\nabla}_{e_{1}}e_{2} &=&0,  \notag \\
\tilde{\nabla}_{e_{2}}e_{2} &=&a(s)b(s)e_{1}-b(s)e_{4},  \notag \\
\tilde{\nabla}_{e_{1}}e_{3} &=&d(s)e_{1}-a(s)d(s)e_{4},  \notag \\
\tilde{\nabla}_{e_{2}}e_{3} &=&0,  \notag \\
\tilde{\nabla}_{e_{1}}e_{4} &=&c(s)e_{1}-a(s)d(s)e_{3},  \notag \\
\tilde{\nabla}_{e_{2}}e_{4} &=&b(s)e_{2},  \notag
\end{eqnarray}%
where 
\begin{eqnarray}
a(s) &=&\frac{x_{3}^{\prime }(s)}{\sqrt{1+\left( x_{3}{}^{\prime }\right)
		^{2}}},\text{ } \\
b(s) &=&\frac{\sqrt{1+\left( x_{3}{}^{\prime }\right) ^{2}}}{x_{3}(s)},\text{
} \\
c(s) &=&\frac{x_{3}^{\prime \prime }(s)}{\sqrt{1+\left( x_{3}{}^{\prime
		}\right) ^{2}}}, \\
d(s) &=&\frac{x_{1}^{\prime \prime }(s)x_{2}^{\prime }(s)-x_{2}^{\prime
		\prime }(s)x_{1}^{\prime }(s)}{\sqrt{1+\left( x_{3}{}^{\prime }\right) ^{2}}}%
.
\end{eqnarray}%
By using (3), (4) and (6), the mean curvature vector and Gaussian curvature
of the surface $M_{1}$ are obtained as:%
\begin{equation}
H=\frac{1}{2}\left( -d(s)e_{3}+\left( c(s)+b\left( s\right) \right)
e_{4}\right)
\end{equation}%
and 
\begin{equation}
K=c(s)b\left( s\right) ,
\end{equation}%
respectively.

By using (2) and (7) the Laplacian of the Gauss map of $M_{1}$\ is computed
as:%
\begin{equation}
\Delta G=L(s)\left( e_{1}\wedge e_{2}\right) +M(s)\left( e_{2}\wedge
e_{3}\right) +N(s)\left( e_{2}\wedge e_{4}\right) ,
\end{equation}%
where 
\begin{equation}
L(s)=d^{2}(s)-b^{2}\left( s\right) -c^{2}\left( s\right) ,
\end{equation}%
\begin{equation}
M(s)=d^{\prime }\left( s\right) +a(s)d(s)(b(s)+c(s)),
\end{equation}%
\begin{equation}
N(s)=b^{\prime }(s)+c^{\prime }(s)+a(s)d^{2}(s).
\end{equation}

\begin{theorem}
	\label{teo 1}Let $M_{1}$ be rotation surface of elliptic type given by the
	parametrization (5). If $M_{1}$ has harmonic Gauss map then it has constant
	Gaussian curvature.
\end{theorem}

\begin{proof}
	Let the Gauss map of $M_{1}$be harmonic , i.e., $\Delta G=0.$ In that case
	from (14), (15), (16) and (17) we have 
	\begin{eqnarray}
	d^{2}(s)-b^{2}\left( s\right) -c^{2}\left( s\right) &=&0, \\
	d^{\prime }\left( s\right) +a(s)d(s)(b(s)+c(s)) &=&0,  \notag \\
	b^{\prime }(s)+c^{\prime }(s)+a(s)d^{2}(s) &=&0.  \notag
	\end{eqnarray}%
	By multiplying both sides of second equation of (18) with $d(s)$ and using
	the third equation of (18) we have that%
	\begin{equation}
	d(s)d^{\prime }\left( s\right) -b(s)b^{\prime }\left( s\right)
	-c(s)c^{\prime }\left( s\right) =(b(s)c(s))^{\prime }.
	\end{equation}
	
	By evaluating the derivative of the first equation of (18) with respect to $%
	s $ and using (19), we have that $b(s)c(s)=$constant and from (13) it
	implies that $K=K_{0}=$constant.
\end{proof}

\begin{theorem}
	\label{teo 2}Let $M_{1}$ be the flat rotation surface of elliptic type given
	by the parameterization (5). Then $M_{1}$ has pointwise 1-type Gauss map if
	and only if the profile curve of $M_{1}$ is characterized in one of the
	following way:
	
	i)%
	\begin{eqnarray}
	x_{1}(s) &=&-\frac{1}{\delta _{1}}\sin \left( -\delta _{1}s+\delta
	_{2}\right) +\delta _{4}, \\
	x_{2}(s) &=&\frac{1}{\delta _{1}}\cos \left( -\delta _{1}s+\delta
	_{2}\right) +\delta _{4},  \notag \\
	x_{3}(s) &=&\delta _{3},  \notag
	\end{eqnarray}%
	where $\delta _{1},$\ $\delta _{2},$\ $\delta _{3}$\ and $\delta _{4}$ are
	real constants and the Gauss map of $M_{1}$\ holds (1) for $f=\delta
	_{1}^{2}-\frac{1}{\delta _{3}^{2}}$ and $C=0.$ If $\delta _{1}\delta
	_{3}=\pm 1$ then the function $f$ becomes zero and it implies that the Gauss
	map is harmonic.
	
	ii)%
	\begin{eqnarray}
	x_{1}(s) &=&\int \left( 1+\lambda _{1}^{2}\right) ^{\frac{1}{2}}\cos \left( -%
	\frac{\lambda _{3}}{\lambda _{1}\left( 1+\lambda _{1}^{2}\right) ^{\frac{1}{2%
	}}}\ln (\lambda _{1}s+\lambda _{2})+\lambda _{4}\right) ds,  \notag \\
	x_{2}(s) &=&\int \left( 1+\lambda _{1}^{2}\right) ^{\frac{1}{2}}\sin \left( -%
	\frac{\lambda _{3}}{\lambda _{1}\left( 1+\lambda _{1}^{2}\right) ^{\frac{1}{2%
	}}}\ln (\lambda _{1}s+\lambda _{2})+\lambda _{4}\right) ds,  \notag \\
	x_{3}(s) &=&\lambda _{1}s+\lambda _{2},
	\end{eqnarray}%
	where $\lambda _{1},$ $\lambda _{2},$ $\lambda _{3}$ and $\lambda _{4}$ are
	real constants and the Gauss map of $M_{1}$holds (1) for $f(s)=\frac{1}{%
		\left( \lambda _{1}s+\lambda _{2}\right) ^{2}}\left( \frac{\lambda _{3}^{2}}{%
		1+\lambda _{1}^{2}}-1\right) $ and $C=\lambda _{1}^{2}e_{1}\wedge
	e_{2}+\lambda _{1}\left( 1+\lambda _{1}^{2}\right) ^{\frac{1}{2}}e_{2}\wedge
	e_{4}.$
\end{theorem}

\begin{proof}
	We suppose that $M_{1}$ has pointwise 1-type Gauss map. By using (1) and
	(14), we get 
	\begin{eqnarray}
	-f+f\left \langle C,e_{1}\wedge e_{2}\right \rangle  &=&-L(s), \\
	f\left \langle C,e_{2}\wedge e_{3}\right \rangle  &=&-M(s),  \notag \\
	f\left \langle C,e_{2}\wedge e_{4}\right \rangle  &=&N(s)  \notag
	\end{eqnarray}%
	and 
	\begin{equation}
	\left \langle C,e_{1}\wedge e_{3}\right \rangle =\left \langle C,e_{1}\wedge
	e_{4}\right \rangle =\left \langle C,e_{3}\wedge e_{4}\right \rangle =0.
	\end{equation}%
	By taking the derivatives of all equations in (23) with respect to $e_{2}$
	and using (22) we obtain 
	\begin{eqnarray}
	a(s)N(s)-L(s)+f &=&0, \\
	a(s)M(s) &=&0,  \notag \\
	M(s) &=&0,  \notag
	\end{eqnarray}%
	respectively. From above equations, we have two cases. One of them is $a(s)=0
	$, $M(s)=0$ and the other is $a(s)\neq 0,$ $M(s)=0.$ Firstly, we suppose
	that $a(s)=0$ and $M(s)=0.$ By using (8), we have that $x_{3}(s)=\delta _{3}$%
	=constant. It implies that $c(s)=0,$ $b\left( s\right) =\frac{1}{\delta _{3}}
	$ and $M_{1}$ is flat. Since the profile curve $x$\ is spacelike curve which
	is parameterized by arc-length, we can put%
	\begin{equation}
	x_{1}^{\prime }(s)=\cos \delta \left( s\right) \text{ (or resp. }\sin \delta
	\left( s\right) \text{ )and }x_{2}^{\prime }(s)=\sin \delta \left( s\right) 
	\text{ (or resp. }\cos \delta \left( s\right) \text{)},
	\end{equation}%
	where $\delta $ is smooth angle function. Without loss of generality we
	assume that 
	\begin{equation*}
	x_{1}^{\prime }(s)=\cos \delta \left( s\right) \text{ and }x_{2}^{\prime
	}(s)=\sin \delta \left( s\right) 
	\end{equation*}
	We can do similar computations for the another case, too. By using third
	equation of (24)\ and (16) we obtain that 
	\begin{equation}
	d\left( s\right) =\delta _{1},\text{\ }\delta _{1}\text{\ is non zero
		constant.}
	\end{equation}%
	On the other hand by using (11), (25) and (26) we get%
	\begin{equation}
	\delta \left( s\right) =-\delta _{1}s+\delta _{2},
	\end{equation}%
	where $\delta _{1},$ $\delta _{2}$ are real constants. Then by substituting
	(27) into (25) and taking the integral we have the equation (20). Also the
	Laplacian of the Gauss map of $M_{1}$ with the equations $a(s)=0,$ $b\left(
	s\right) =\frac{1}{\delta _{3}}$, $c(s)=0$ and $d\left( s\right) =\delta _{1}
	$ is found as $\Delta G=\left( \delta _{1}^{2}-\frac{1}{\delta _{3}^{2}}%
	\right) G$
	
	Now we suppose that $a(s)\neq 0$ and $M(s)=0.$ Since the surface $M_{1}$ is
	flat, i.e., $K=0.$ By using (13) we have that $c(s)=0.$ From (10) we get 
	\begin{equation}
	x_{3}(s)=\lambda _{1}s+\lambda _{2}
	\end{equation}%
	for some constants $\lambda _{1}\neq 0$ and $\lambda _{2}.$ In that case by
	using (8), (9) and (28) we have 
	\begin{equation}
	a(s)=\frac{\lambda _{1}}{\left( 1+\lambda _{1}^{2}\right) ^{\frac{1}{2}}}
	\end{equation}%
	and 
	\begin{equation}
	b(s)=\frac{\left( 1+\lambda _{1}^{2}\right) ^{\frac{1}{2}}}{\lambda
		_{1}s+\lambda _{2}}.
	\end{equation}%
	Let consider that $M(s)=0$ with $c(s)=0$. In that case from (16), we obtain
	that%
	\begin{equation}
	d^{\prime }\left( s\right) +a(s)b(s)d(s)=0
	\end{equation}%
	By using (29), (30) and (31) we have 
	\begin{equation}
	d(s)=\frac{\lambda _{3}}{\lambda _{1}s+\lambda _{2}},
	\end{equation}%
	where $\lambda _{3}$\ is constant of integration. On the other hand, Since
	the profile curve $x$\ is spacelike curve which is parameterized by
	arc-length, we can put%
	\begin{eqnarray}
	x_{1}^{\prime }(s) &=&\left( 1+\lambda _{1}^{2}\right) ^{\frac{1}{2}}\cos
	\lambda \left( s\right) , \\
	x_{2}^{\prime }(s) &=&\left( 1+\lambda _{1}^{2}\right) ^{\frac{1}{2}}\sin
	\lambda \left( s\right) ,  \notag
	\end{eqnarray}%
	where $\lambda $ is smooth angle function. By differentiating (33) we obtain 
	\begin{eqnarray}
	x_{1}^{\prime \prime }(s) &=&-\left( 1+\lambda _{1}^{2}\right) ^{\frac{1}{2}%
	}\sin \lambda \left( s\right) \lambda ^{\prime }\left( s\right) , \\
	x_{2}^{\prime \prime }(s) &=&\left( 1+\lambda _{1}^{2}\right) ^{\frac{1}{2}%
	}\cos \lambda \left( s\right) \lambda ^{\prime }\left( s\right) .  \notag
	\end{eqnarray}%
	By using (11), (28), (33) and (34)$,$ we get%
	\begin{equation}
	d(s)=-\left( 1+\lambda _{1}^{2}\right) ^{\frac{1}{2}}\lambda ^{\prime
	}\left( s\right) .
	\end{equation}%
	By combining (32) and (35) we obtain 
	\begin{equation}
	\lambda \left( s\right) =-\frac{\lambda _{3}}{\lambda _{1}\left( 1+\lambda
		_{1}^{2}\right) ^{\frac{1}{2}}}\ln (\lambda _{1}s+\lambda _{2})+\lambda _{4}.
	\end{equation}%
	So by substituting (36) into (33), we get 
	\begin{eqnarray}
	x_{1}(s) &=&\int \left( 1+\lambda _{1}^{2}\right) ^{\frac{1}{2}}\cos \left( -%
	\frac{\lambda _{3}}{\lambda _{1}\left( 1+\lambda _{1}^{2}\right) ^{\frac{1}{2%
	}}}\ln (\lambda _{1}s+\lambda _{2})+\lambda _{4}\right) ds,  \notag \\
	x_{2}(s) &=&\int \left( 1+\lambda _{1}^{2}\right) ^{\frac{1}{2}}\sin \left( -%
	\frac{\lambda _{3}}{\lambda _{1}\left( 1+\lambda _{1}^{2}\right) ^{\frac{1}{2%
	}}}\ln (\lambda _{1}s+\lambda _{2})+\lambda _{4}\right) ds,  \notag
	\end{eqnarray}
	
	Conversely, the surface $M_{1}$ whose the profil curve given by (21) is
	pointwise 1-type Gauss map for%
	\begin{equation*}
	f(s)=\frac{1}{\left( \lambda _{1}s+\lambda _{2}\right) ^{2}}\left( \frac{%
		\lambda _{3}^{2}}{1+\lambda _{1}^{2}}-1\right)
	\end{equation*}%
	and 
	\begin{equation*}
	C=\lambda _{1}^{2}e_{1}\wedge e_{2}+\lambda _{1}\left( 1+\lambda
	_{1}^{2}\right) ^{\frac{1}{2}}e_{2}\wedge e_{4}.
	\end{equation*}
\end{proof}

\begin{theorem}
	\label{teo 3} A non- minimal rotational surfaces of elliptic type $M_{1}$
	defined by (5) has pointwise 1-type Gauss map of the first kind if and only
	if the mean curvature vector of $M_{1}$ is parallel .
\end{theorem}

\begin{proof}
	From (12) we have that $H=\frac{1}{2}\left( -d(s)e_{3}+\left( c(s)+b\left(
	s\right) \right) e_{4}\right) .$ Let the mean curvature vector of $M_{1}$ be
	parallel , i.e., $DH=0.$ Then we get 
	\begin{equation*}
	D_{e_{1}}H=\frac{1}{2}\left( -M(s)e_{3}+N(s)e_{4}\right) =0.
	\end{equation*}%
	In this case we obtain that $M(s)=N(s)=0.$ From (14), we have that $\Delta
	G=L(s)e_{1}\wedge e_{2}.$
	
	Conversely, if $M_{1}$ has pointwise 1-type Gauss map of the first kind then
	from (14) we get $M(s)=N(s)=0$ and it implies that $M_{1}$ has parallel mean
	curvature vector.
\end{proof}

\begin{corollary}
	\label{cor 1}If rotational surfaces of elliptic type $M_{1}$ given by (5) is
	minimal then it has pointwise 1-type Gauss map of the first kind.
\end{corollary}

\subsection{Rotational surfaces of hyperbolic type with pointwise 1-type
	Gauss map in $\mathbb{E}_{2}^{4}$}

In this subsection, firstly we consider rotational surfaces of hyperbolic
type with harmonic Gauss map. Further we obtain a characterization of flat
rotational surfaces of hyperbolic type with pointwise 1-type Gauss map and
give a relationship for non-minimal these surfaces with parallel mean
curvature vector and pointwise 1-type Gauss map of the first kind. The
proofs of theorems in this subsection are similar the proofs of theorems in
previous section so we give the theorems as without proof.

Rotational surface of hyperbolic type $M_{2}$ is defined by%
\begin{equation*}
\varphi \left( t,s\right) =%
\begin{pmatrix}
\cosh t & 0 & \sinh t & 0 \\ 
0 & 1 & 0 & 0 \\ 
\sinh t & 0 & \cosh t & 0 \\ 
0 & 0 & 0 & 1%
\end{pmatrix}%
\left( 
\begin{array}{c}
x_{1}(s) \\ 
x_{2}(s) \\ 
0 \\ 
x_{4}(s)%
\end{array}%
\right)
\end{equation*}%
\begin{equation}
M_{2}:\text{ }\varphi \left( t,s\right) =\left( x_{1}(s)\cosh
t,x_{2}(s),x_{1}(s)\sinh t,x_{4}(s)\right) ,
\end{equation}%
where the surface $M_{2}$ is obtained by the rotation of the curve $%
x(s)=(x_{1}(s),x_{2}(s),0,x_{4}(s))$ about the two dimensional Euclidean
plane span$\left \{ \epsilon _{2},\epsilon _{4}\right \} .$ Let the profile
curve of $M_{2}$ be unit speed spacelike curve. In that case $\left(
x_{1}{}^{\prime }(s)\right) ^{2}+\left( x_{2}{}^{\prime }(s)\right)
^{2}-\left( x_{4}{}^{\prime }(s)\right) ^{2}=1$. We assume that $x_{1}(s)>0.$
The moving frame field $\left \{ e_{1},e_{2},e_{3},e_{4}\right \} $ on $%
M_{2} $\ is choosen as follows:%
\begin{eqnarray*}
	e_{1} &=&\left( x_{1}{}^{\prime }(s)\cosh t,x_{2}{}^{\prime
	}(s),x_{1}{}^{\prime }(s)\sinh t,x_{4}{}^{\prime }(s)\right) , \\
	e_{2} &=&\left( \sinh t,0,\cosh t,0\right) , \\
	e_{3} &=&\frac{1}{\sqrt{\varepsilon \left( \left( x_{1}{}^{\prime }\right)
			^{2}-1\right) }}\left( 0,x_{4}{}^{\prime }(s),0,x_{2}{}^{\prime }(s)\right) ,
	\\
	e_{4} &=&\frac{1}{\sqrt{\varepsilon \left( \left( x_{1}{}^{\prime }\right)
			^{2}-1\right) }}\left( \left( 1-\left( x_{1}{}^{\prime }\right) ^{2}\right)
	\cosh t,-x_{1}{}^{\prime }(s)x_{2}{}^{\prime }(s),\left( 1-\left(
	x_{1}{}^{\prime }\right) ^{2}\right) \sinh t,-x_{1}{}^{\prime
	}(s)x_{4}{}^{\prime }(s)\right) ,
\end{eqnarray*}%
where $e_{1},e_{2}$ and $e_{3},e_{4}$ are tangent vector fields and normal
vector fields to $M_{2},$ respectively and $\varepsilon $\ is signature of $%
\left( x_{1}{}^{\prime }\right) ^{2}-1.$ If $\left( x_{1}{}^{\prime }\right)
^{2}-1$ is positive (resp. negative) then $\varepsilon =1$ (resp. $%
\varepsilon =-1$). It is easily seen that 
\begin{equation*}
\left \langle e_{1},e_{1}\right \rangle =-\left \langle e_{2},e_{2}\right
\rangle =1,\text{ }\left \langle e_{3},e_{3}\right \rangle =-\left \langle
e_{4},e_{4}\right \rangle =\varepsilon .
\end{equation*}%
we have the dual 1-forms as: 
\begin{equation*}
\omega _{1}=ds\text{ \  \  \  \ and \  \  \  \ }\omega _{2}=-x_{1}(s)dt.
\end{equation*}%
After some computations, components of the second fundamental form and the
connection forms are obtained by:%
\begin{eqnarray}
h_{11}^{3} &=&d(s),\ h_{12}^{3}=0,\ h_{22}^{3}=0, \\
h_{11}^{4} &=&c(s),\text{ \ }h_{12}^{4}=0,\text{ \ }h_{22}^{4}=-\varepsilon
b(s)  \notag
\end{eqnarray}%
and%
\begin{eqnarray*}
	\omega _{12} &=&a(s)b(s)\omega _{2},\text{ \  \ }\omega _{13}=d(s)\omega _{1},%
	\text{ \  \ }\omega _{14}=c(s)\omega _{1}, \\
	\omega _{23} &=&0,\text{ \  \ }\omega _{24}=\varepsilon b(s)\omega _{2},\text{
		\  \ }\omega _{34}=a(s)d(s)\omega _{1}.
\end{eqnarray*}%
The covariant differentiations with respect to $e_{1}$ and $e_{2}$ are
computed as: 
\begin{eqnarray}
\tilde{\nabla}_{e_{1}}e_{1} &=&\varepsilon d(s)e_{3}-\varepsilon c(s)e_{4} \\
\tilde{\nabla}_{e_{2}}e_{1} &=&a(s)b(s)e_{2}  \notag \\
\tilde{\nabla}_{e_{1}}e_{2} &=&0  \notag \\
\tilde{\nabla}_{e_{2}}e_{2} &=&a(s)b(s)e_{1}+b(s)e_{4}  \notag \\
\tilde{\nabla}_{e_{1}}e_{3} &=&-d(s)e_{1}-\varepsilon a(s)d(s)e_{4}  \notag
\\
\tilde{\nabla}_{e_{2}}e_{3} &=&0  \notag \\
\tilde{\nabla}_{e_{1}}e_{4} &=&-c(s)e_{1}-\varepsilon a(s)d(s)e_{3}  \notag
\\
\tilde{\nabla}_{e_{2}}e_{4} &=&-\varepsilon b(s)e_{2}  \notag
\end{eqnarray}%
where 
\begin{equation*}
a(s)=\frac{x_{1}^{\prime }(s)}{\sqrt{\varepsilon \left( \left(
		x_{1}{}^{\prime }\right) ^{2}-1\right) }},
\end{equation*}%
\begin{equation*}
b(s)=\frac{\sqrt{\varepsilon \left( \left( x_{1}{}^{\prime }\right)
		^{2}-1\right) }}{x_{1}(s)},
\end{equation*}%
\begin{equation*}
c(s)=\frac{x_{1}^{\prime \prime }(s)}{\sqrt{\varepsilon \left( \left(
		x_{1}{}^{\prime }\right) ^{2}-1\right) }},
\end{equation*}%
\begin{equation*}
d(s)=\frac{x_{2}^{\prime \prime }(s)x_{4}^{\prime }(s)-x_{4}^{\prime \prime
	}(s)x_{2}^{\prime }(s)}{\sqrt{\varepsilon \left( \left( x_{1}{}^{\prime
		}\right) ^{2}-1\right) }}.
\end{equation*}%
By using (3), (4) and (38), the mean curvature vector and Gaussian curvature
of the surface $M_{2}$ are obtained as follows:%
\begin{equation*}
H=\frac{1}{2}\left( \varepsilon d(s)e_{3}-\varepsilon \left(
c(s)+\varepsilon b\left( s\right) \right) e_{4}\right)
\end{equation*}%
and 
\begin{equation*}
K=c(s)b\left( s\right) ,
\end{equation*}%
respectively.

By using (2) and (39) the Laplacian of the Gauss map of $M_{2}$\ is computed
as:%
\begin{equation*}
\Delta G=L(s)\left( e_{1}\wedge e_{2}\right) +M(s)\left( e_{2}\wedge
e_{3}\right) +N(s)\left( e_{2}\wedge e_{4}\right) ,
\end{equation*}%
where 
\begin{equation*}
L(s)=\varepsilon \left( d^{2}(s)-c^{2}\left( s\right) -b^{2}\left( s\right)
\right) ,
\end{equation*}%
\begin{equation*}
M(s)=\varepsilon \left( d^{\prime }\left( s\right) +\varepsilon
a(s)d(s)\left( c(s)+\varepsilon b\left( s\right) \right) \right) ,
\end{equation*}%
\begin{equation*}
N(s)=-\varepsilon \left( c^{\prime }(s)+\varepsilon b^{\prime
}(s)+\varepsilon a(s)d^{2}(s)\right) .
\end{equation*}

\begin{theorem}
	\label{teo 4}Let $M_{2}$ be rotation surface of hyperbolic type given by the
	parameterization (37). If $M_{2}$\ has Gauss map harmonic then it has
	constant Gaussian curvatrure.
\end{theorem}

\begin{theorem}
	\label{teo 5}Let $M_{2}$ be flat rotation surface of hyperbolic type given
	by the parameterization (37). Then $M_{2}$ has pointwise 1-type Gauss map if
	and only if the profile curve of $M_{2}$ is characterized in one of the
	following way:
	
	i)%
	\begin{eqnarray*}
		x_{1}(s) &=&\delta _{1}, \\
		x_{2}(s) &=&-\frac{1}{\delta _{2}}\sinh \left( -\delta _{2}s+\delta
		_{3}\right) +\delta _{4}, \\
		x_{4}(s) &=&-\frac{1}{\delta _{2}}\cosh \left( -\delta _{2}s+\delta
		_{3}\right) +\delta _{4},
	\end{eqnarray*}
	
	where $\delta _{1},$\ $\delta _{2},$\ $\delta _{3}$\ and $\delta _{4}$ are
	real constants and the Gauss map $G$ holds (1) for $f=\frac{1}{\delta
		_{1}^{2}}-\delta _{2}^{2}$ and $C=0.$ If $\delta _{1}\delta _{2}=\pm 1$ then
	the function $f$ becomes zero and it implies that the Gauss map is harmonic.
	
	ii)%
	\begin{eqnarray*}
		x_{1}(s) &=&\lambda _{1}s+\lambda _{2}, \\
		x_{2}(s) &=&\int \left( \lambda _{1}^{2}-1\right) ^{\frac{1}{2}}\sinh \left( 
		\frac{\lambda _{3}}{\lambda _{1}\left( \lambda _{1}^{2}-1\right) ^{\frac{1}{2%
		}}}\ln (\lambda _{1}s+\lambda _{2})+\lambda _{4}\right) ds, \\
		x_{4}(s) &=&\int \left( \lambda _{1}^{2}-1\right) ^{\frac{1}{2}}\cosh \left( 
		\frac{\lambda _{3}}{\lambda _{1}\left( \lambda _{1}^{2}-1\right) ^{\frac{1}{2%
		}}}\ln (\lambda _{1}s+\lambda _{2})+\lambda _{4}\right) ds,
	\end{eqnarray*}%
	where $\lambda _{1},$ $\lambda _{2},$ $\lambda _{3}$ and $\lambda _{4}$ are
	real constants and without loss of generality we suppose that $\lambda
	_{1}^{2}-1>0.$ Morever the Gauss map $G$ holds (1) for the function $f(s)=%
	\frac{1}{\left( \lambda _{1}s+\lambda _{2}\right) ^{2}}\left( 1-\frac{%
		\lambda _{3}^{2}}{\lambda _{1}^{2}-1}\right) $ and $C=-\lambda
	_{1}^{2}e_{1}\wedge e_{2}+\lambda _{1}\left( \lambda _{1}^{2}-1\right) ^{%
		\frac{1}{2}}e_{2}\wedge e_{4}.$
\end{theorem}

\begin{theorem}
	\label{teo 6} A non- minimal rotational surfaces of hyperbolic type $M_{2}$
	defined by (37) has pointwise 1-type Gauss map of the first kind if and only
	if $M_{2}$ has parallel mean curvature vector
\end{theorem}

\begin{corollary}
	\label{cor 2}If rotational surfaces of hyperbolic type $M_{2}$ given by (37)
	is minimal then it has pointwise 1-type Gauss map of the first kind.
\end{corollary}

\subsection{Rotational surfaces of parabolic type with pointwise 1-type
	Gauss map in $\mathbb{E}_{2}^{4}$}

In this subsection, we study rotational surfaces of parabolic type with
pointwise 1-type Gauss map. We show that flat rotational surface of
parabolic type has pointwise 1-type Gauss map if and only if its Gauss map
is harmonic. Also we conclude that flat rotational surface of parabolic type
has harmonic Gauss map if and only if it has parallel mean curvature vector.

We consider the pseudo-orthonormal base $\left \{ \epsilon _{1},\xi _{2},\xi
_{3},\epsilon _{4}\right \} $ of $\mathbb{E}_{2}^{4}$\ such that $\xi _{2}=%
\frac{\epsilon _{2}+\epsilon _{3}}{\sqrt{2}},$\ $\xi _{3}=\frac{-\epsilon
	_{2}+\epsilon _{3}}{\sqrt{2}}$ $\left \langle \xi _{2},\xi _{2}\right \rangle
=\left \langle \xi _{3},\xi _{3}\right \rangle =0$ and $\left \langle \xi
_{2},\xi _{3}\right \rangle =-1.$ Let consider $\alpha $ spacelike curve is
given by%
\begin{equation*}
x\left( s\right) =x_{1}(s)\epsilon _{1}+x_{2}(s)\epsilon
_{2}+x_{3}(s)\epsilon _{3}
\end{equation*}%
or we can express $x$\ according to pseudo-orthonormal base $\left \{
\epsilon _{1},\xi _{2},\xi _{3},\epsilon _{4}\right \} $ as follows:%
\begin{equation*}
x\left( s\right) =x_{1}(s)\epsilon _{1}+p(s)\xi _{2}+q(s)\xi _{3},
\end{equation*}%
where $p(s)=\frac{x_{2}(s)+x_{3}(s)}{\sqrt{2}}$ and $q(s)=\frac{%
	-x_{2}(s)+x_{3}(s)}{\sqrt{2}}.$\ The rotational surface of parabolic type $%
M_{3}$ is defined by%
\begin{equation}
M_{3}:\varphi \left( t,s\right) =x_{1}(s)\epsilon _{1}+p(s)\xi
_{2}+(-t^{2}p(s)+q(s))\xi _{3}+\sqrt{2}tp(s)\epsilon _{4},
\end{equation}%
We suppose that $x$\ is parameterized by arc-length, that is, $\left(
x_{1}{}^{\prime }(s)\right) ^{2}-2p^{\prime }(s)q^{\prime }(s)=1.$ Now we
can give a moving orthonormal frame $\left \{ e_{1},e_{2},e_{3},e_{4}\right \} 
$\ for $M_{3}$ as follows:%
\begin{eqnarray*}
	e_{1} &=&x_{1}{}^{\prime }(s)\epsilon _{1}+p^{\prime }(s)\xi
	_{2}+(-t^{2}p^{\prime }(s)+q^{\prime }(s))\xi _{3}+\sqrt{2}tp^{\prime
	}(s)\epsilon _{4}, \\
	e_{2} &=&-\sqrt{2}t\xi _{3}+\epsilon _{4}, \\
	e_{3} &=&\epsilon _{1}+\frac{x_{1}{}^{\prime }(s)}{p^{\prime }(s)}\xi _{3},
	\\
	e_{4} &=&x_{1}{}^{\prime }(s)\epsilon _{1}+p^{\prime }(s)\xi _{2}+(\frac{1}{%
		p^{\prime }(s)}+q^{\prime }(s)-t^{2}p^{\prime }(s))\xi _{3}+\sqrt{2}%
	tp^{\prime }(s)\epsilon _{4},
\end{eqnarray*}%
where $p^{\prime }(s)$ is non zero. Then it is easily seen that 
\begin{equation*}
\left \langle e_{1},e_{1}\right \rangle =\left \langle e_{3},e_{3}\right \rangle
=1,\text{ }\left \langle e_{2},e_{2}\right \rangle =\left \langle
e_{4},e_{4}\right \rangle =-1.
\end{equation*}%
We have the dual 1-forms as: 
\begin{equation*}
\omega _{1}=ds\text{ \  \  \  \ and \  \  \  \ }\omega _{2}=-\sqrt{2}p\left(
s\right) dt.
\end{equation*}%
Also we obtain components of the second fundamental form and the connection
forms as:%
\begin{eqnarray}
h_{11}^{3} &=&c(s),\ h_{12}^{3}=0,\ h_{22}^{3}=0, \\
h_{11}^{4} &=&-b(s),\text{ \ }h_{12}^{4}=0,\text{ \ }h_{22}^{4}=a(s)  \notag
\end{eqnarray}%
and%
\begin{eqnarray*}
	\omega _{12} &=&a(s)\omega _{2},\text{ \  \ }\omega _{13}=c(s)\omega _{1},%
	\text{ \  \ }\omega _{14}=-b(s)\omega _{1}, \\
	\omega _{23} &=&0,\text{ \  \ }\omega _{24}=-a(s)\omega _{2},\text{ \  \ }%
	\omega _{34}=-c(s)\omega _{1}.
\end{eqnarray*}%
Then we get the covariant differentiations with respect to $e_{1}$ and $e_{2}
$ as follows:%
\begin{eqnarray}
\tilde{\nabla}_{e_{1}}e_{1} &=&c(s)e_{3}+b(s)e_{4}, \\
\tilde{\nabla}_{e_{2}}e_{1} &=&a(s)e_{2},  \notag \\
\tilde{\nabla}_{e_{1}}e_{2} &=&0,  \notag \\
\tilde{\nabla}_{e_{2}}e_{2} &=&a(s)e_{1}-a(s)e_{4},  \notag \\
\tilde{\nabla}_{e_{1}}e_{3} &=&-c(s)e_{1}+c(s)e_{4},  \notag \\
\tilde{\nabla}_{e_{2}}e_{3} &=&0,  \notag \\
\tilde{\nabla}_{e_{1}}e_{4} &=&b(s)e_{1}+c(s)e_{3},  \notag \\
\tilde{\nabla}_{e_{2}}e_{4} &=&a(s)e_{2},  \notag
\end{eqnarray}%
where 
\begin{equation}
a(s)=\frac{p^{\prime }(s)}{p(s)},
\end{equation}%
\begin{equation}
b(s)=\frac{p^{\prime \prime }(s)}{p^{\prime }(s)},
\end{equation}%
\begin{equation}
c(s)=\frac{x_{1}^{\prime \prime }(s)p^{\prime }(s)-p^{\prime \prime
	}(s)x_{1}^{\prime }(s)}{p^{\prime }(s)}.
\end{equation}%
By using (3), (4) and (41), the mean curvature vector and Gaussian curvature
of the surface $M_{3}$ are obtained as follows:%
\begin{equation}
H=\frac{1}{2}\left( c(s)e_{3}+\left( a(s)+b\left( s\right) \right)
e_{4}\right) 
\end{equation}%
and 
\begin{equation}
K=a(s)b\left( s\right) ,
\end{equation}%
respectively.

By using (2) and (42) the Laplacian of the Gauss map of $M_{3}$\ is computed
as:%
\begin{equation}
\Delta G=L(s)\left( e_{1}\wedge e_{2}\right) +M(s)\left( e_{2}\wedge
e_{3}\right) +N(s)\left( e_{2}\wedge e_{4}\right) ,
\end{equation}%
where 
\begin{equation}
L(s)=c^{2}(s)-a^{2}\left( s\right) -b^{2}\left( s\right) ,
\end{equation}%
\begin{equation}
M(s)=c^{\prime }\left( s\right) +c(s)(a(s)+b(s)),
\end{equation}%
\begin{equation}
N(s)=c^{2}(s)+a^{\prime }\left( s\right) +b^{\prime }\left( s\right) .
\end{equation}

\begin{theorem}
	\label{teo 7}Let $M_{3}$ be flat rotation surface of parabolic type given by
	the parameterization (40). Then $M_{3}$ has pointwise 1-type Gauss map if
	and only if the profile curve of $M_{3}$ is given by%
	\begin{eqnarray*}
		x_{1}(s) &=&\frac{\varepsilon }{\mu _{1}}\left( \ln (\mu _{1}s+\mu _{2})(\mu
		_{1}s+\mu _{2})\right) +\left( \mu _{4}-\varepsilon \right) s+\mu _{5}, \\
		p(s) &=&\mu _{1}s+\mu _{2}, \\
		q(s) &=&\frac{1}{2\mu _{1}}\int \left( \left( \varepsilon \ln \left( \mu
		_{1}s+\mu _{2}\right) +\mu _{4}\right) ^{2}-1\right) ds,
	\end{eqnarray*}%
	where $\mu _{1},$ $\mu _{2},$ $\mu _{4},$ $\mu _{5}$ real constants. Morever
	the surface $M_{3}$ has harmonic Gauss map for $f=0.$
\end{theorem}

\begin{proof}
	We suppose that $M_{3}$ has pointwise 1-type Gauss map. In that case the
	Gauss map of $M_{3}$\ holds (1). By using (1) and (48), we get 
	\begin{eqnarray}
	-f+f\left \langle C,e_{1}\wedge e_{2}\right \rangle  &=&-L(s), \\
	f\left \langle C,e_{2}\wedge e_{3}\right \rangle  &=&-M(s),  \notag \\
	f\left \langle C,e_{2}\wedge e_{4}\right \rangle  &=&N(s)  \notag
	\end{eqnarray}%
	and 
	\begin{equation}
	\left \langle C,e_{1}\wedge e_{3}\right \rangle =\left \langle C,e_{1}\wedge
	e_{4}\right \rangle =\left \langle C,e_{3}\wedge e_{4}\right \rangle =0.
	\end{equation}%
	By taking the derivatives of all equations in (53) with respect to $e_{2}$
	and using (52) we obtain 
	\begin{eqnarray}
	L(s)-N(s) &=&f, \\
	M(s) &=&0,  \notag
	\end{eqnarray}%
	respectively$.$ Since the surface $M_{3}$ is flat, i.e., $K=0$ from (47) we
	have that $b(s)=0$. From (44) we obtain that 
	\begin{equation}
	p(s)=\mu _{1}s+\mu _{2}
	\end{equation}%
	for some constants $\mu _{1}\neq 0$ and $\mu _{2}.$ By using (43) and (55)
	we have that 
	\begin{equation}
	a(s)=\frac{\mu _{1}}{\mu _{1}s+\mu _{2}}.
	\end{equation}%
	If we consider $M(s)=0$ with the equations $b(s)=0$ and $a(s)=\frac{\mu _{1}%
	}{\mu _{1}s+\mu _{2}},$ from (50) we get 
	\begin{equation}
	c(s)=\frac{\mu _{3}}{\mu _{1}s+\mu _{2}}.
	\end{equation}%
	On the other hand, by using the first equation of (54), (49), (51), (56) and
	(57) we obtain that $f=0$. It means that $L(s)=N(s)=0$ and we have%
	\begin{equation*}
	\mu _{3}=\varepsilon \mu _{1},\  \varepsilon =\pm 1.
	\end{equation*}%
	If we consider (45), (55) and (57) we get 
	\begin{equation}
	x_{1}(s)=\frac{\varepsilon }{\mu _{1}}\left( \ln (\mu _{1}s+\mu _{2})(\mu
	_{1}s+\mu _{2})\right) +\left( \mu _{4}-\varepsilon \right) s+\mu _{5},
	\end{equation}%
	where $\mu _{4},$ $\mu _{5}$ are constants of integration. Since $x$ is unit
	speed spacelike curve we get%
	\begin{equation}
	q^{\prime }(s)=\frac{\left( x_{1}{}^{\prime }(s)\right) ^{2}-1}{2p^{\prime
		}(s)}.
	\end{equation}%
	By substituting (55) and (58) into (59) we obtain%
	\begin{equation*}
	q(s)=\frac{1}{2\mu _{1}}\int \left( \left( \varepsilon \ln \left( \mu
	_{1}s+\mu _{2}\right) +\mu _{4}\right) ^{2}-1\right) ds.
	\end{equation*}%
	This completes the proof.
\end{proof}

\begin{theorem}
	\label{teo 8} Let $M_{3}$ be flat rotational surfaces of parabolic type
	given by (40). $M_{3}$ has harmonic\ Gauss map if and only if its mean
	curvature vector is parallel.
\end{theorem}

\begin{proof}
	We suppose that $M_{3}$ has parallel mean curvature vector, i.e., $DH=0.$
	From (46) we have that 
	\begin{equation*}
	D_{e_{1}}H=\frac{1}{2}\left( M(s)e_{3}+N(s)e_{4}\right) =0.
	\end{equation*}%
	In this case we obtain that $M(s)=N(s)=0.$ Since $M_{3}$ is a flat surface,
	from the previous theorem we have 
	\begin{equation*}
	b(s)=0\text{ and }a(s)=\frac{\mu _{1}}{\mu _{1}s+\mu _{2}}.
	\end{equation*}%
	By considering the equation $M(s)=0$ with above equations and using (50) we
	get 
	\begin{equation*}
	c(s)=\frac{\mu _{3}}{\mu _{1}s+\mu _{2}},
	\end{equation*}%
	where $\mu _{3}$\ is the constant of integration. It implies that $L(s)=0.$
	Hence we obtain that Gauss map of $M_{3}$ is harmonic .
	
	Conversely, if $M_{3}$ is harmonic then it is easily seen that $DH=0.$
\end{proof}

The first author is supported by Ahi Evran University :PYO-EGF.4001.15.002.

\end{document}